\documentclass[12pt,draft]{amsart}
\usepackage[colorinlistoftodos]{todonotes}
\usepackage{graphicx}
\usepackage{amssymb}
\newtheorem{theorem}{Theorem}
\newtheorem{lemma}[theorem]{Lemma}
\newtheorem{prop}[theorem]{Proposition}
\newtheorem{corollary}[theorem]{Corollary}

\theoremstyle{definition}
\newtheorem{definition}[theorem]{Definition}

\theoremstyle{remark}
\newtheorem{remark}[theorem]{Remark}

\newcommand{\DD}{{\mathbb D}}

\newcommand{\bL}{{\mathbb L}}

\newcommand{\NN}{{\mathbb N}}

\newcommand{\GG}{{\mathbb G}}
\newcommand{\EE}{{\mathbb E}}

\newcommand{\RR}{{\mathbb R}}

\newcommand{\CC}{{\mathbb C}}

\newcommand{\TT}{{\mathbb T}}

\DeclareMathOperator{\aut}{Aut}
\DeclareMathOperator{\id}{id}
\DeclareMathOperator{\Id}{I}
\DeclareMathOperator{\GL}{GL}

\DeclareMathOperator{\SO}{SO}

\renewcommand{\phi}{\varphi}

\begin{document}
\title{A family of Lempert domains}
\author{Armen Edigarian}
\address{Faculty of Mathematics and Computer Science \\
     Jagiellonian University\\
     Prof. St. \L ojasiewicza 6\\
     30-348 Krak\'ow, Poland \\ 
orcid: 0000-0003-1789-1287}

\begin{abstract}
In \cite{G-Z} G.~Ghosh and W. Zwonek introduced a new class of domains $\bL_n$, $n\ge1$, which are 2-proper holomorphic images of the Cartan domains of type four. This family contains biholomorphic images of the symmetrized bidisc and the tetrablock. It is well-known, that symmetrized bidisc and tetrablock are Lempert type domains. In our paper we show that the whole family of domains $\bL_n$ are Lempert domains.
\end{abstract}


\maketitle

\section{Introduction}
Let $X$ be a complex manifold. In many applications, it is natural to consider on $X$ a distance $d_X$ related to the complex structure and consider $(X,d_X)$ as a metric space. It turns out that there are many natural ways to introduce these distances. For a good reference of this subject, see e.g. \cite{J-P1}. The largest one is the \emph{Lempert function}, related to the Kobayashi pseudodistance. We denote by $\DD$ the unit disc and by $\TT$ the unit circle in the complex plane $\CC$. For simplicity, we restrict ourselves to domains in the space $\CC^n$. Let $D\subset\CC^n$ be a domain and let $z,w\in D$ be any points. The Lempert function is defined as
$$
\ell_D(z,w)=\inf\{\rho(0,\sigma):f:\DD\to D\text{ holomorphic},\\ f(0)=z, f(\sigma)=w\},
$$
where $\rho$ is the Poincar\'e distance in the disc (see \cite{J-P1}, Chapter III).

On the other hand, the smallest one is the \emph{Carath\'eodory pseudodistance} (see \cite{J-P1}, Chapter II). We put
$$
c_D(z,w)=\sup\{\rho(f(z),f(w)): f:D\to\DD\text{ holomorphic}\}, \quad z,w\in D.
$$
It is well-known that for the unit ball, products of the unit balls, and some modifications, these two functions agree, i.e., there is only one natural way to introduce a distance on these domains. More generally, it is true on classical Cartan domains,i.e., bounded symmetric homogeneous domains in $\CC^n$ (see \cite{Hua}). The proof uses that these domains are homogeneous and the Schwarz lemma type results at the origin. Domains, on which we have this property, sometimes are called \emph{Lempert domains}.

In 1981 L. Lempert \cite{Lem} proved a fundamental result that for a convex domain $D\subset\CC^n$ there is an equality
\begin{equation}\label{eq:Lem}
c_D=\ell_D.
\end{equation}
Later it was extended by L. Lempert to bounded strongly linearly convex pseudoconvex domains (see \cite{Lem2}).

For more than 20 years it was an open question, whether there are other types of Lempert domains, besides biholomorphic images of convex ones and some simple modifications.
In a series of papers \cite{A-Y1, A-Y2, A-Y3} J.~Agler and N.J.~Young introduced and analysed the symmetrized bidisc, which was the first known bounded hyperconvex domain for which we have the equality of the Carath\'eodory distance and Lempert function but which cannot be exhausted by domains biholomorphic to convex domains. The symmetrized bidisc $\GG_2$ is the image of the bidisc $\DD^2$ under the symmetrization mapping
$$
\DD^2\ni(\lambda_1,\lambda_2)\to (\lambda_1+\lambda_2,\lambda_1\lambda_2)\in\CC^2.
$$

In \cite{A-W-Y} the authors introduced another domain, tetrablock in $\CC^3$, which has the same properties (for the definition and the properties of the tetrablock see below).

Recently G.~Ghosh and W.~Zwo\-nek (see \cite{G-Z}) proposed a new class of domains, which includes symmetrized bidisc and tetrablock. Our main aim is to show that for domains in this class equality \eqref{eq:Lem} holds, so they are Lempert domains.

Let $\langle z,w\rangle=\sum_{j=1}^n z_j\bar w_j$
for $z,w\in\mathbb{C}^n$ be the Hermitian 
inner product in $\CC^n$ and let $\|z\|^2=\langle z,z\rangle$ be the corresponding norm. We define another norm in $\CC^nC$, as follows
$$
p(z)^2=\|z\|^2+\sqrt{\|z\|^4-|\langle z,\bar z\rangle|^2}, \quad z\in\CC^n.
$$
The unit ball in this norm we denote by $L_n=\{z\in\CC^n: p(z)<1\}$ and call
it the Cartan domain of type four (sometimes it is called the Lie ball of dimension $n\ge1$). It is a bounded symmetric homogeneous domain (see \cite{C1}, \cite{C2}).

Following \cite{G-Z} we define a 2-proper holomorphic mapping
$$
\Lambda_n:\CC^n\ni (z_1,z_2,\dots,z_n)\mapsto
(z_1^2,z_2,\dots,z_n)\in\CC^n
$$
and put $\bL_n=\Lambda_n(L_n)$.

The main result of the paper is the following.
\begin{theorem}\label{thm:1!} For any $z,w\in\bL_n$ we have
$$
c_{\bL_n}(z,w)=\ell_{\bL_n}(z,w).
$$
\end{theorem}

As we mentioned above, Theorem~\ref{thm:1!} was known for $n=2$ (symmetrized bidisc) and for $n=3$ (tetrablock \see{E-K-Z}). So, for any $n\ge4$ it provide us with example of non-trivial (i.e., product of domains biholomorphic to convex ones) non-convex domains so that the equality \eqref{eq:Lem} holds. Moreover, our proof shows that these domains, being a modification of the homogeneous symmetric domains, in a sense, are quite natural in relation with the invariant distances and metrics.

\section{The group of automorphisms}
In this part we describe automorphisms of the Cartan domain of type four in $\CC^n$ and give some useful properties. For more information see \cite{Fab}, \cite{Hua}, \cite{Morita}.

For $z\in\CC^n$ we put
$$
a(z)=\sqrt{\frac{\|z\|^2+|\langle z,\bar z\rangle|}{2}}\quad\text{ and }\quad b(z)=\sqrt{\frac{\|z\|^2-|\langle z,\bar z\rangle|}{2}}.
$$
In \cite{Abate1}, page 278, the numbers $a(z)$ and $b(z)$ are called the modules related to the Cartan domain of type four. Note that
$p(z)=a(z)+b(z)$ for any $z\in\CC^n$. For a point $z\in\CC^n$ we have $p(z)<1$ if and only if
\begin{equation}\label{eq:3}
\|z\|<1\quad\text{ and }\quad 2\|z\|^2<1+|\langle z,\bar z\rangle|^2.
\end{equation}
From this inequality we have.
\begin{lemma} Let $z=(z_1,z')\in\CC\times\CC^{n-1}$. Then $z\in\bL_n$ if and only if 
\begin{equation}\label{eq:4}
|z_1|+\|z'\|^2<1\quad\text{ and }\quad 2|z_1|+2\|z'\|^2<1+|z_1+\langle z',\bar z'\rangle|^2.
\end{equation}
\end{lemma}
Note that if $(z_1,\dots,z_{n-1})\in L_{n-1}$ then 
$(z_1,\dots,z_{n-1},0)\in L_{n}$.
Let us show the following simple  inequality.
\begin{lemma} Let $z=(z_1,z')\in\CC\times\CC^{n-1}$.
Assume that $\|z\|<1$.
Then
$$
2\|z'\|^2-|\langle z',\bar z'\rangle|^2\le
2\|z\|^2-|\langle z,\bar z\rangle|^2.
$$
Moreover, the equality holds if and only if $z_1=0$.
\end{lemma}
\begin{proof}
We have to show
$|\langle z,\bar z\rangle|^2-
|\langle z',\bar z'\rangle|^2\le2|z_1|^2$.
This is equivalent with
$
|z_1|^4+2\Re\big(\bar z_1^2\langle z',\bar z')\big)\le2|z_1|^2$.
It suffices to show
$|z_1|^4+2|z_1|^2|\langle z',\bar z'\rangle|\le2|z_1|^2$.
But we have
$|z_1|^2+2|\langle z',\bar z'\rangle|\le
|z_1|^2+2\|z'\|^2<2$.
\end{proof}
Assume that $z=(z_1,\dots,z_n)\in L_n$. Then from the above Lemma we have $(z_1,\dots,z_{n-1})\in L_{n-1}$.

In all our considerations, the group of real special orthogonal matrices $\SO_n(\RR)$, i.e., real matrices $A$ such that $A^T A=A A^T=\Id_n$ and $\det A=1$, is essencial. This follows from the properties: $\langle Az,Aw\rangle=\langle z,w\rangle$ and
 $\langle Az,\overline{Aw}\rangle=\langle z,\bar w\rangle$
for any $z,w\in\CC^n$ and any $A\in\SO_n(\RR)$. In particular, $p(Az)=p(z)$ for any $z\in\CC^n$ and any $A\in\SO_n(\RR)$.

Let $z\in\CC^n$. Then there exists an $\eta\in\TT$
such that $\eta^2\langle z,\bar z\rangle=|\langle z,\bar z\rangle|$. Note that
 $\eta^2\langle z,\bar z\rangle=\langle \eta z,\overline{\eta z}\rangle$. 
If $\eta z=u+iv$, where $u,v\in\RR^n$, then vectors $u$ and $v$ are orthogonal in $\RR^n$ and  $\|u\|\ge\|v\|$. 
Easy calculations show that $\|u\|=a(z)$ and $\|v\|=b(z)$.
So, there exists a matrix $A\in\SO_n(\RR)$ (a 'rotation' of $\RR^n$) such that $Au=(a(z),0,\dots0)$ and $Av=(0,b(z),\dots,0)$. We get $\eta Az=A(\eta z)=(a(z),b(z)i,\dots,0)$. We formulate the above considerations as a separate result.
\begin{lemma}
For any $z\in\CC^n$, $n\ge2$, there exists a matrix $A\in\SO_n(\RR)$ and a number $\eta\in\TT$ such that
$$
\eta Az=
A(\eta z)=(a(z),b(z)i,0,\dots,0).
$$
\end{lemma}
Fixing first $k$-coordinates and "rotating" last $n-k$ coordinates we have
\begin{lemma}
For any $z\in\CC^n$, $n\ge3$, and any $k\in\NN$ with $k\le n$, there exists a matrix $A\in\SO_n(\RR)$ and a number $\eta\in\TT$ such that
$Aw\in\{(w_1,\dots,w_k)\}\times\CC^{n-k}$ for any $w=(w_1,\dots,w_n)\in\CC^n$ and
$$
Az=(z_1,\dots,z_k,\eta a(z'),\eta b(z')i,0,\dots,0),
$$
where $z'=(z_{k+1},\dots,z_n)\in\CC^{n-k}$.
\end{lemma}

Note that for any $a,b\in\RR$ we have $(a,bi,0,\dots,0)\in L_n$ if and only if $|a|+|b|<1$.

Recall the following well-known result.
\begin{theorem} Let $\Phi:L_n\to L_n$ be a biholomorphic mapping such that $\Phi(0)=0$. Then
there exist a matrix $A\in\SO_d(\RR)$ and a number $\eta\in\TT$ such that $\Phi(z)=\eta Az$ for any $z\in L_n$.
\end{theorem}
Automorphisms of $L_n$ described in the above Theorem are called linear automorphisms. Now we are going to describe all automorphisms of the irreducible classical Cartan domain of type four (see \cite{Hua}, \cite{Fab}, \cite{Morita}).
We define first a group of matrices.
\begin{multline}
G(n)=\Big\{g=
\begin{bmatrix}
A & B\\
C & D
\end{bmatrix}:
A\in\GL(n,\RR),
B\in M(n;2;\RR),\\
C\in M(2,n;\RR),
D\in\GL(2,\RR), \det D>0,\\ g^t
\begin{bmatrix}
\Id & 0\\
0 & -\Id_2
\end{bmatrix}g=
\begin{bmatrix}
\Id & 0\\
0 & -\Id_2
\end{bmatrix}
\Big\}.
\end{multline}
For any $g\in G$ we define the following $\CC^n$-valued holomorphic function
$$
\Psi_g(z)=\frac{Az+BW(z)}{(1\ i)(Cz+DW(z))},
$$
where 
$W(z)=\begin{bmatrix} \frac12(\langle z,\bar z\rangle+1) \\
\frac{i}2(\langle z,z\rangle-1)
\end{bmatrix}$.
One can show (see \cite{Fab}) that for any $g\in G$ we have
$$
(1\ i)(Cz+DW(z))\not=0\quad\text{ for all }z\in\overline{L_d}.
$$
Thus, $\Psi_g$ is well-defined on $\overline{L_n}$.
Note that we have a homomorphism of groups
$\kappa:G(n)\to G(n+1)$
defined by 
$$
\kappa(g)=
\begin{bmatrix}
1 & 0 & 0\\
0 & A & B\\
0 & C & D
\end{bmatrix},
$$
where $g=\begin{bmatrix}
A & B\\
C & D
\end{bmatrix}$.
As a simple corollary we get.
\begin{corollary}\label{cor:5}
For any $\Psi\in\aut(L_n)$ there exists a $\tilde\Psi\in\aut(L_{n+1})$ such that $\tilde\Psi(0,z)=(0,\Psi(z))$ for any $z\in L_n$.
\end{corollary}
\begin{proof}
Assume that $\Psi=\Psi_g$ for some $g\in G(n)$. Put $\tilde\Psi=\Psi_{\kappa(g)}$.
\end{proof}
\begin{remark} Changing the coordinates we may take $\tilde\Psi(z,0)=(\Psi(z),0)$.
\end{remark}

In \cite{G-Z} the authors gave a description of automorphisms of the domain $\bL_n$. Using this description and results above we get.
\begin{theorem} Let $\Phi\in\aut(\bL_{n})$. Then there exists a $g\in G(n-1)$ such that 
$$
\Lambda_{n}\circ\Psi_{\kappa(g)}=\Phi\circ\Lambda_{n}.
$$
Moreover, for any $z\in L_{n-1}$ we have $\Phi(0,z)=(0,\Psi_g(z))$.
\end{theorem}
From Corollary~\ref{cor:5} we get (cf. \cite{G-Z}, Theorem 5.3).
\begin{theorem}
Let $\Phi\in\aut(\bL_n)$. Then there exists a $\tilde\Phi\in\aut(\bL_{n+1})$ such that $\tilde\Phi(z,0)=(\Phi(z),0)$ for any $z\in\bL_n$.
\end{theorem}
From the definition and similar property for $L_n$ we have.
\begin{lemma}
$\Pi:\bL_n\ni(z_1,\dots,z_n)\to (z_1,\dots,z_{n-1})\in\bL_{n-1}$ and
$Q:\bL_{n-1}\ni(z_1,\dots,z_{n-1})\to (z_1,\dots,z_{n-1},0)\in\bL_{n}$
are well-defined holomorphic mappings such that $\Pi\circ Q=\id_{\bL_{n-1}}$.
\end{lemma}

\begin{corollary}\label{cor:12} Let $n\ge3$.
We have holomorphic mappings $Q:\bL_3\to\bL_n$ and $\Pi:\bL_n\to\bL_3$ such that $\Pi\circ Q=\id_{\bL_3}$.
In particular, for any $z,w\in Q(\bL_3)$ we have 
$$
c_{\bL_n}(z;w)=\ell_{\bL_n}(z;w)\quad\text{ for any }z,w\in\bL_n.
$$
\end{corollary}

\section{The properties of the tetrablock and of the domain $\bL_3$}
In our paper properties of the tetrablock and a domain $\bL_3$ are crucial. First recall the following definition (see \cite{A-W-Y}, Definition 1.1).
\begin{definition}
The \emph{tetrablock} is the domain
$$
\EE=\{z\in\CC^3: 1-\lambda_1 z_1-\lambda_2 z_2+\lambda_1\lambda_2 z_3\not=0 \text{ whenever } \lambda_1,\lambda_2\in\overline{\DD}\}.
$$
\end{definition}

In \cite{A-W-Y} the authors gave several equivalent characterizations of the tetrablock. One of them is the following (see \cite{A-W-Y}, Theorem 2.2).
\begin{prop}
For $z\in\CC^3$ we have: $z\in\EE$ if and only if
$$
|z_1|^2+|z_2|^2+2|z_1z_2-z_3|<1+|z_3|^2\text{ and }
|z_3|<1.
$$
\end{prop}

From this Proposition, definition of the Cartan domain, and the domain $\bL_n$ we have  the following biholomorphism (see \cite{G-Z}, Corollary 3.5).
\begin{lemma}\label{lemma:tetrablock} Put
$$
\Psi(z_1,z_2,z_3)=(z_2+iz_3,z_2-iz_3,z_1+z_2^2+z_3^2).
$$
Then $\Psi:\bL_3\to\EE$ is a biholomorphic mapping such that
$\Psi(r,0,0)=(0,0,r)$ for any $r\in[0,1)$.
\end{lemma}

Recall the following result (see \cite{Young}, Theorem 5.2).
\begin{theorem}\label{thm:Young}
Let $z\in\EE$ be any point. Then there exist an automorphism $\Phi$ of $\EE$ and a number $r\in[0,1)$ such that $\Phi(z)=(r,0,0)$.
\end{theorem}

From Theorem~\ref{thm:Young} and Lemma~\ref{lemma:tetrablock}  we get.
\begin{corollary}
For any $z\in\bL_3$ there exists an automorphism of $\bL_3$ such that $\Phi(z)=(r,0,0)$, where $r\in[0,1)$.
\end{corollary}

From this we get the following important result.
\begin{corollary}
For any $z\in\bL_n$ there exists an automorphism $\Phi$ of $\bL_n$ such that $\Phi(z)=(\rho,0,\dots,0)$, where $\rho\in[0;1)$.
\end{corollary}

All these analyses and remarks imply the following crucial result.
\begin{theorem}\label{thm:19}
For any $z,w\in\bL_n$ there exists an automorphism $\Phi$ of $\bL_n$ such that $\Phi(z)=(\rho,0,\dots,0)$ and $\Phi(w)\in\bL_3\times\{0\}_{n-3}$.
\end{theorem}

\begin{proof}[Proof of Theorem~\ref{thm:1!}] The proof follows from Theorem~\ref{thm:19} and Corollary~\ref{cor:12}.
\end{proof}

\section{Invariant distences and metrics in $\bL_n$}
Let $D\subset\CC^n$ be a domain. For a point $z\in D$ and a vector $X\in\CC^n$, we recall that the infinitesimal Kobayashi metric at $z$ in the direction $X$ is defined to be 
\begin{multline*}
\kappa_{D}(z;X)=\sup\{\alpha>0: f:\DD\to D\text{ holomorphic},\\ f(0)=z, f'(0)=\frac{1}{\alpha}X\}.
\end{multline*}

We show the following.
\begin{theorem} Let $n\ge3$.
$$
\kappa_{\bL_n}(0;X)=|X_1|+p_{d-1}(X'),
$$
where $X=(X_1,X')\in\CC\times\CC^{n-1}$.
\end{theorem}

\begin{proof} 
We know that (see \cite{Young}, Theorem 2.1)
$$
\kappa_{\EE}(0;X)=\max\{|X_1|,|X_2|\}+|X_3|\quad\text{ for any }X\in\CC^3.
$$
From the biholomorphicity between $\EE$ and $\bL_3$ we infer 
$$
\kappa_{\bL_3}(0;(X_1,X_2,X_3))=|X_1|+\max\{|X_2+iX_3|,|X_2-iX_3|\}.
$$
For the general case, 
there exists an $\eta\in\TT$ and a matrix 
$A\in\SO_n(\RR)$ such that
$$
AX=(X_1,\eta a(X'),\eta ib(X'),0,\dots,0).
$$
Take an automorphism $\Phi$ of $\bL_n$ generated by $A$. Then $\Phi(0)=0$ and
$$
\Phi'(0)X=(X_1,\eta a(X'),i\eta b(X'),0,\dots,0).
$$
Then
$$
\kappa_{\bL_n}(0;X)=\kappa_{\bL_n}(0;\Phi'(0)X)=\kappa_{\bL_3}(0;(X_1,\eta a(X'),\eta ib(X')),
$$
and, therefore,
$$
\kappa_{\bL_n}(0;X)=|X_1|+a(X')+b(X')=|X_1|+p(X').
$$
\end{proof}

Recall also the following result (see \cite{A-W-Y}, Corollary 3.7).
\begin{theorem}
For any $z=(z_1,z_2,z_3)\in\EE$ with $|z_1|\le |z_2|$, we have
$$
c_{\EE}(0;z)=\tanh^{-1} \frac{|z_2-\bar z_1z_3|+|z_1z_2-z_3|}{1-|z_1|^2}.
$$
\end{theorem}
As a corollary of this we get.
\begin{corollary}
Let $z=(z_1,z')\in\bL_n$, where $z_1\in\CC$ and $z'\in\CC^{n-1}$. If $z'=0$ then
$c_{\bL_n}(0;z)=\tanh^{-1}(|z_1|)$. If $z'\not=0$ then we have
$$
c_{\bL_n}(0;z)=
\tanh^{-1}\Big(p(z')\Big|1-\frac{z_1\langle \bar z',z'\rangle}
{p^2(z')-|\langle z',\bar z'\rangle|^2}\Big|
+\frac{|z_1|p(z')^2}{p^2(z')-|\langle z',\bar z'\rangle|^2}\Big).
$$
\end{corollary}
\begin{proof}
In a similar way as above, if $|z_2+iz_3|\le |z_2-iz_3|$ then
\begin{multline*}
c_{\bL_3}(0, (z_1,z_2,z_3))=\tanh^{-1}\Big(\big|z_2-iz_3-\frac{z_1 \overline{(z_2+iz_3)}}{1-|z_2+iz_3|^2}\big|\\
+\frac{|z_1|}{1-|z_2+iz_3|^2}\Big).
\end{multline*}
If $z\in\bL_n$, then using appropriate automorphism, we may assume that $z=(z_1,\bar\eta a(z'),\bar\eta b(z')i,\dots,0)$ and, therefore,
\begin{multline*}
c_{\bL_n}(0;z)=\tanh^{-1}\Big(\big|\bar\eta p(z')-\frac{\eta z_1 (a(z')-b(z'))}{1-|a(z')-b(z')|^2}\big|\\
+\frac{|z_1|}{1-|a(z')-b(z')|^2}\Big),
\end{multline*}
where $\eta\in\TT$ is such that $\eta^2\langle z',\bar z'\rangle=
|\langle z',\bar z'\rangle|$. From this we obtain the formula.
\end{proof}

\begin{remark}
In a forthcoming paper, using similar technique, we study complex geodesics in $\bL_n$ (see \cite{Edi2}).
\end{remark}

\end{document}